\documentclass[11pt]{amsart}
\usepackage[utf8]{inputenc}
\usepackage{bm}
\usepackage{fontenc}
\usepackage{amsfonts}
\usepackage{amssymb}
\usepackage{amsmath}
\usepackage{amsthm}\usepackage{mathtools}
\usepackage{enumerate}
\usepackage{hyperref}\usepackage{enumitem}
\usepackage{mathrsfs}
\usepackage{tikz}\usetikzlibrary{calc}
\usepackage{marginnote}
\usepackage{xcolor,enumitem}
\usepackage{soul}\usepackage{tikz}
\usepackage[all,cmtip]{xy}

\newcommand{\Z}{\mathbb{Z}}

\newcommand{\N}{\mathbb{N}}
\newcommand{\C}{\mathbb{C}}

\newcommand{\eps}{\varepsilon}

\newtheorem*{acknowledgments}{Acknowledgments}

\theoremstyle{case1}

\theoremstyle{case2}

\theoremstyle{case3}

\newcommand{\SOTh}{\mathrm{SOT}\text{-}}

\newcommand{\cstu}{\mathrm{C}^*_u}
\newcommand{\cstql}{\mathrm{C}^*_{ql}}

\newtheorem*{rigprob*}{Rigidity Problem for uniform Roe Algebras}
\newtheorem*{rigprobcorona*}{Rigidity Problem for uniform Roe Coronas}

\newcommand{\cst}{\mathrm{C}^*}
\newcommand{\cstar}{$\mathrm{C}^*$}

\newcommand{\cU}{\mathcal{U}}

\newcommand{\cB}{\mathcal{B}}
\newcommand{\cK}{\mathcal{K}}

\newtheorem{theorem}{Theorem}[section]
\newtheorem*{theorem*}{Theorem}
\newtheorem{proposition}[theorem]{Proposition}
\newtheorem{problem}[theorem]{Problem}
\newtheorem*{proposition*}{Proposition}
\newtheorem{lemma}[theorem]{Lemma}
\newtheorem*{lemma*}{Lemma}
\newtheorem{corollary}[theorem]{Corollary}
\newtheorem*{corollary*}{Corollar}

\newtheorem*{fact*}{Fact}
\theoremstyle{definition}
\newtheorem{definition}[theorem]{Definition}
\newtheorem*{definition*}{Definition}
\newtheorem{claim}[theorem]{Claim}
\newtheorem*{claim*}{Claim}

\newtheorem*{conjecture*}{Conjecture}

\theoremstyle{remark}

\newtheorem*{example*}{Example}
\newtheorem{remark}[theorem]{Remark}
\newtheorem*{remark*}{Remark}

\newtheorem*{note*}{Note}
\newtheorem*{question*}{Question}

\DeclareMathOperator{\Span}{span}

\DeclareMathOperator{\supp}{supp}

\DeclareMathOperator{\propg}{prop}

\DeclareMathOperator{\Ad}{Ad}

\DeclareMathOperator{\diam}{diam}

\newcommand{\rM}{\mathrm{ M}}

\newcounter{my_enumerate_counter}
\newcommand{\pushcounter}{\setcounter{my_enumerate_counter}{\value{enumi}}}
\newcommand{\popcounter}{\setcounter{enumi}{\value{my_enumerate_counter}}}

\usepackage{enumitem}

\begin{document}

\title[Coarse quotients and uniform Roe algebras ]{Coarse quotients of metric spaces and embeddings of uniform Roe algebras}%
\author[B. M. Braga]{Bruno M. Braga}
\address{University of Virginia, 141 Cabell Drive, Kerchof Hall, P.O. Box 400137, Charlottesville, USA}
\email{demendoncabraga@gmail.com}
\urladdr{https://sites.google.com/site/demendoncabraga}

\subjclass[2010]{}
\keywords{}
\thanks{ }
\date{\today}%

\begin{abstract}
We study embeddings of uniform Roe algebras which have ``large range'' in their codomain and the relation of those with coarse quotients  between metric spaces.  Among other results, we show that if $Y$ has   property A and there is an embedding $\Phi:\cstu(X)\to \cstu(Y)$ with ``large range'' and so that $\Phi(\ell_\infty(X))$ is a Cartan subalgebra of $\cstu(Y)$, then there is a bijective coarse quotient $X\to Y$. This shows that the large scale geometry of $Y$ is, in some sense, controlled by the one of $X$. For instance, if $X$ has finite asymptotic dimension, so does $Y$.
\end{abstract}
\maketitle


\section{Introduction}\label{SecIntro}

Given a metric space $X$, one associates to it a \cstar-algebra $\cstu(X)$, called the \emph{uniform Roe algebra of $X$}, which captures some of the large scale geometric properties of $X$ (see Section \ref{SecPrelim} for precise definitions regarding uniform Roe algebras and coarse geometry). This algebra was introduced by J. Roe in the context of index theory of elliptical operators in noncompact manifolds \cite{Roe1988,Roe1993}, but its uses have  spread way beyond this field. As a sample application, their study   	 entered the realm of mathematical physics in the context of topological materials and, in particular, of topological insulators. For instance,  Y. Kubota has proposed the uniform Roe  
 algebra of $\Z^n$ as a model for disordered topological materials while   E. Ewert and  R. Meyer have used its non uniform version for the same purpose (see \cite{Kubota2017,EwertMeyer2019}).   In this context, in order to study various types of symmetries, it is important to better understand embeddings/isomorphisms of (uniform) Roe
algebra.

This paper deals with preservation of the large scale geometry of uniformly locally finite metric spaces\footnote{Recall, a metric space $(X,d)$ is \emph{uniformly locally finite} if given $r>0$ there is $N\in\N$ so that the balls of radius $r$ in $X$ contains at most $N$-many elements. Finitely generated groups and $k$-regular graphs are examples of such spaces.} under embeddings between their uniform Roe algebras  --- the  study of such embeddings was  initiated by  I.  Farah, A. Vignati and the current author in \cite{BragaFarahVignati2019}. Firstly, notice that an injective coarse map $X\to Y$ induces a canonical embedding $\cstu(X)\to \cstu(Y)$:   given such $f:X\to Y$, the isometric embedding   $u_f:\ell_2(X)\to \ell_2(Y)$ determined  by $u_f\delta_x=\delta_{f(x)}$ gives the embedding $\Ad(u_f):\cstu(X)\to \cstu(Y)$ (\cite[Theorem 1.2]{BragaFarahVignati2019}). If $f$ is, furthermore,   a coarse embedding, then the image of $\Phi$ is a hereditary subalgebra of $\cstu(Y)$ and, if $f$ is a bijective coarse equivalence, then $\Phi$ is an isomorphism (see Figure \eqref{Fig1}).

 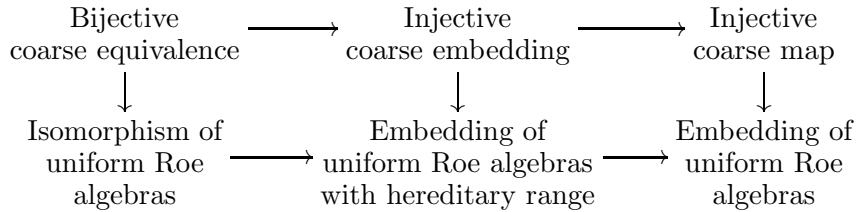
\begin{figure*}[h] \centerline{%
 \xymatrix@R-2ex{ 
    \txt{Bijective \\ coarse equivalence} \ar[r] \ar[d]    & \txt{Injective \\ coarse embedding  }  \ar[r]\ar[d] & \txt{Injective \\ coarse map} \ar[d] \\
        \txt{Isomorphism of \\ uniform Roe\\ algebras   } \ar[r]& \txt{Embedding of \\ uniform Roe algebras \\ with hereditary range  } \ar[r] & \txt{Embedding of \\ uniform Roe\\ algebras}   }
}
\caption{ Relation between coarse properties of maps between  metric spaces and embeddings of their uniform Roe algebras}\label{Fig1}
\end{figure*}

Rigidity questions for uniform Roe algebras deal  with  when the vertical arrows in Figure \eqref{Fig1} can be reversed. The first two arrows are known to be reversable if the codomain space has G. Yu's property A (see \cite[Theorem 6.13]{WhiteWillett2017} and  \cite[Theorem 1.4]{BragaFarahVignati2019}).  As for the latter,  the following holds: under some geometric conditions on $Y$,  (1)  a rank-preserving embedding $\cstu(X)\to \cstu(Y)$ gives a uniformly finite-to-one\footnote{Recall,   a map  $f:X\to Y$ is   \emph{uniformly finite-to-one} if $\sup_{y\in Y}|f^{-1}(n)|<\infty$.  } coarse map $X\to Y$ (\cite[Theorem 5.4]{BragaFarahVignati2019}), and (2) a compact-preserving   embedding $\cstu(X)\to \cstu(Y)$ gives a partition for  $X$   into finitely many pieces 	 all of which can be mapped into $Y$ by an injective coarse map (\cite[Theorem 1.2]{BragaFarahVignati2019}). Although those results do not give  injective coarse maps, their power   comes from the fact that, for  uniformly locally finite metric spaces,  the existence of   uniformly finite-to-one coarse maps $X\to Y$  often implies that $X$ inherits large scale geometric properties of $Y$ --- among them, we have  property A, asymptotic dimension and finite   decomposition complexity (\cite[Corollary 1.3]{BragaFarahVignati2019}). 

The current  paper focus  on better understanding the form taken by  embeddings $\cstu(X)\to \cstu(Y)$   and on obtaining results on geometry preservation which are ``opposite'' to the ones mentioned above. Precisely, we are interested in when the existence of an  embedding $\cstu(X)\to \cstu(Y)$  can make $Y$ inherit geometric properties of $X$ (instead of $X$ inheriting properties of $Y$). Obviously, we exclude the case in which the embedding $\Phi:\cstu(X)\to \cstu(Y)$ is an isomorphism from our  range of interests. Indeed, if that is so, not only $\Phi^{-1}:\cstu(Y)\to \cstu(X)$ is an embedding and the previous results apply, but one can also obtain a coarse equivalence between $X$ and $Y$ under some geometric conditions  (see \cite{SpakulaWillett2013,BragaFarah2018} for precise statements). We are then interested in non isomorphic embeddings $\cstu(X)\to \cstu(Y)$ which have ``large enough  range''; forcing  then  the geometry of $Y$ to be ``controlled'' by the one of $X$.
 
Before giving the definition of  an embedding with  ``large enough  range'', we introduce the concept which motivated it. Recall,  a surjective bounded linear map  between Banach spaces is called a \emph{quotient map}. By the open mapping theorem, the image of the unit ball under quotient maps contains a ball with positive radius. With that in mind, S. Zhang extended in \cite{Zhang2015Israel}  the concept of quotient maps to the (coarse) metric   space category  as follows:  let $X$ and $Y$ be metric spaces and $f:X\to Y$. A coarse  map  $f$ is a \emph{coarse quotient map} if there is $K>0$ such that  for all $\eps>0$ there is $\delta>0$ so that 
 \[B(f(x),\eps)\subset f(B(x,\delta))^K\]
 for all $x\in X$, where $f(B(x,\delta))^K$ is the $K$-neighborhood of   $f(B(x,\delta))$.\footnote{When extending   quotient maps to   metric spaces, the absense of linearity  gives rise to   two distinct   approaches: extend the concept taking into  account the (1) large scale geometry of the metric spaces or (2) the small scale geometry (i.e., its uniform structure). We deal with the former in this paper. For   \emph{uniform quotient maps}, see  \cite{BatesJohnsonLindenstraussPreissSchechtman1999GAFA}.}  Coarse quotients have been studied further in  \cite{BaudierZhang2016JLMS,Zhang2018PAMS} and a weaker notion,  called \emph{weak coarse quotient map}, was studied in \cite{HigginbothamWeighill2019JTopAnal}.
 
 Notice that coarse quotients do not    need to be   coarse embeddings; not even if the map is a bijection. 	We refer   to Subsection \ref{SubsectionCoarseQuotientMapsEX} for examples of coarse quotient maps. For the time being, a quick spoiler:   if $G$ is a finite group acting on a metric space $X$   by coarse equivalences, then the canonical map from  $X$ to the orbit space $ X/ G$ is a coarse quotient map.

As  coarse quotients are objects   in between coarse maps and coarse equivalences, there should be a  diagram as the one in Figure \ref{Fig1} with   coarse quotients appearing in the middle column. More precisely, as it is automatic  from the definition  that a coarse quotient $f:X\to Y$ is   $K$-dense in $Y$ (for $K$ as above), there is not much loss in generality to look at \emph{bijective} coarse quotients.    The next definition introduces the terminology needed. Recall, the propagation of   $a=[a_{xy}]\in \cB(\ell_2(X))$ is defined as 
  \[\propg(a)=\sup\{d(x,y)\mid a_{xy}\neq 0\}.\]
 
\begin{definition}\label{DefiCoboundedEmb}
Let $X$ be a metric space and     $A\subset \cstu(X)$ be a  \cstar-subalgebra.
\begin{enumerate}
\item Let $\eps,k,\ell>0$. An element $b\in \cstu(X)$ is \emph{$(\eps,k,\ell)$-cobounded}   if there are $a_1,\ldots,a_\ell\in A$, with $\|a_i\|\leq \ell$,   and $c_1,\ldots, c_\ell\in \cstu(X)$, with $\propg(c_i)\leq k$ and $\|c_i\|\leq \ell$,  so that  \[\Big\|b-\sum_{i=1}^\ell c_ia_i\Big\|\leq \eps;\]
and $b$ is \emph{$(\eps,k)$-cobounded}  if it is $(\eps,k,\ell)$-cobounded  for some $\ell$. 
\item The subalgebra $A$ is \emph{cobounded in $\cstu(X)$} if there is $k>0$ so that every $b\in \cstu(X)$ is  $(\eps,k)$-cobounded   for all $\eps>0$. 
\item \label{ItemStBound} The subalgebra $A$ is \emph{strongly-cobounded in $\cstu(X)$} if there is $k>0$ so that every contraction $b\in \cstu(X)$ is $(\eps,k,k)$-cobounded  for all $\eps>0$. 
\end{enumerate}
\end{definition}  
 
It is not hard to obtain nonisomorphic  embeddings $\cstu(X)\to \cstu(Y)$ whose images are strongly-cobounded in $\cstu(Y)$  and, in particular, cobounded (see Remark \ref{RemarkCobStCob}).

 We now proceed to describe the main results of this paper. We start giving a  characterization of  the existence of bijective coarse quotients $X\to Y$ in terms of uniform Roe algebras.  Notice that this characterization does not depend on extra geometric conditions on either $X$ or $Y$. 
  
 \begin{theorem}\label{ThmIFFNoGeoProp}
  Let $X$ and $Y$ be uniformly locally finite  metric spaces. The following are equivalent:
\begin{enumerate}
\item There is a bijective coarse quotient map $X\to Y$.
\item There is an embedding $\Phi:\cstu(X)\to \cstu(Y)$ so that $\Phi(\ell_\infty(X))=\ell_\infty(Y)$ and $ \Phi(\cstu(X))$ is  cobounded in $\cstu(Y)$.
\end{enumerate}  
\end{theorem}

Theorem \ref{ThmIFFNoGeoProp} should be compared with two known results: (1)  $X$ and $Y$ are bijectively coarsely equivalent if and only if there is an isomorphism $\Phi:\cstu(X)\to \cstu(Y)$ so that $\Phi(\ell_\infty(X))=\ell_\infty(Y)$ (\cite[Theorem 8.1]{BragaFarah2018}); and  (2) there is an injective coarse map $X\to Y$   if and only if there is an embedding  $\Phi:\cstu(X)\to \cstu(Y)$ so that $\Phi(\ell_\infty(X))\subset \ell_\infty(Y)$ (although not explicitly, this follows easily from \cite[Theorem 4.3 and Lemma 5.3]{BragaFarahVignati2019}).

While Theorem \ref{ThmIFFNoGeoProp} gives  a complete characterization of the existence of bijective coarse quotients, its hypothesis are too rigid. Indeed,  the demand that  $  \Phi(\ell_\infty(X))=\ell_\infty(Y)$ uses too much of the structure given by a choice of basis of $\ell_2(Y)$ (namely, its canonical orthonormal basis).  Hence, it is interesting to obtain results under milder conditions on the embeddings.   
 
 We have the following   in the presence of G. Yu's property A:

 \begin{theorem}\label{ThmEmbIFFCoarseQuotient}
 Let $X$ and $Y$ be uniformly locally finite metric spaces, and assume that $Y$ has property A. If there is an embedding $ \Phi:\cstu(X)\to \cstu(Y)$ so that $ \Phi(\ell_\infty(X))$ is a  Cartan subalgebra of   $\cstu(Y)$ and $\Phi(\cstu(X))$ is   strongly-cobounded in $\cstu(Y)$, then there is a bijective coarse quotient $X\to Y$. 
 \end{theorem}
  
Roughly speaking, property A is needed in Theorem \ref{ThmEmbIFFCoarseQuotient} for three reasons: (1) selecting a map $X\to Y$, (2) assuring the map can be taken to be a bijection, and (3)   to guarantee that $(\cstu(Y),\Phi(\ell_\infty(X)))$ is a \emph{Roe Cartan pair}. The former can actually be obtained under  milder geometric assumptions on $Y$ (see Theorem \ref{ThmEmbCoarseQuotientGHOST}).   As for the latter,   recall that Roe Cartan pairs were introduced by S. White and  R. Willett in \cite{WhiteWillett2017} as follows:   a subalgebra $B\subset \cstu(Y)$ is called \emph{coseparable} if there is a countable $S\subset \cstu(Y)$ so that $\cstu(Y)=\cst(B,S)$. Then $(\cstu(Y),B)$ is a \emph{Roe Cartan pair} if $B$ is a coseparable Cartan subalgebra of $\cstu(Y)$ which is  isomorphic to $\ell_\infty(\N)$ (see Subsection \ref{SubsectionRoeCartan}). It is open if a Cartan subalgebra of $\cstu(Y)$  isomorphic to $\ell_\infty(\N)$ is automatically coseparable. However, it has been recently  shown that this is indeed the case if $Y$ has property A (\cite[Corollary 6.3]{BragaFarahVignati2020}).

  We now   describe   a version of  Theorem \ref{ThmEmbIFFCoarseQuotient} which holds under weaker geometric conditions.   A metric space   $(X,d)$ is  \emph{sparse} if there is   a partition $X=\bigsqcup_n X_n$ into finite subsets so that $d(X_n,X_m)\to \infty$ as $n+m\to \infty$. Also, we say that $X$ \emph{yields only compact ghost projections} if all ghost projections in $\cstu(X)$ are compact.\footnote{Recall, an operator $a=[a_{xy}]\in \cB(\ell_2(X))$ is called a \emph{ghost} if for all $\eps>0$ there is a finite $A\subset X$ so that $|a_{xy}|<\eps$ for all $x,y\not\in A$ (see Subsection \ref{SubsectionURA} for details).} The weaker geometric property which we look at is the one of \emph{all sparse subspaces of $X$ yielding only compact ghost projections}.  This is a fairly broad property: indeed, it is not only implied by  property A, but also by coarse  embeddability  into  $\ell_2$ and, more generally, by the validity of the coarse Baum-Connes conjecture with coefficients  (see \cite[Lemma 7.3]{BragaFarah2018} and \cite[Theorem 5.3]{BragaChungLi2019}).
 
 The following is a version of Theorem \ref{ThmEmbIFFCoarseQuotient} outside the scope of property A. Since property A is not assumed, we loose bijectivity and coseparability of the range is needed in the hypothesis.

\begin{theorem}\label{ThmEmbCoarseQuotientGHOST}
 Let $X$ and $Y$ be uniformly locally finite metric spaces, and assume that all sparse subspaces of $Y$ yield only compact ghost projections. If there is an embedding $ \Phi:\cstu(X)\to \cstu(Y)$ so that $\Phi(\ell_\infty(X))$  is a  coseparable Cartan subalgebra of   $\cstu(Y)$ and $\Phi(\cstu(X))$ is   strongly-cobounded in $\cstu(Y)$, then there is a uniformly finite-to-one coarse quotient $X\to Y$.
 \end{theorem}

  Theorem \ref{ThmEmbCoarseQuotientGHOST}   can be applied to obtain restrictions to the geometry of $Y$. Without getting into   details, we mention that the next result is a  corollary of Theorem  \ref{ThmEmbCoarseQuotientGHOST} and results on the ``coarse  version'' of a function $X\to Y$ being $n$-to-1. Precisely, we show that a uniformly finite-to-one coarse quotient  $X\to Y$ between uniformly locally finite metric spaces  is automatically \emph{coarsely $n$-to-1} for some $n\in\N$  (see  Definition \ref{DefiCoraselyNto1} and Proposition \ref{PropCoarselyNto1}).

  \begin{corollary}\label{CorFiniteAsDim}
   Let $X$ and $Y$ be uniformly locally finite metric spaces, and assume that all sparse subspaces of $Y$ yield only compact ghost projections. Suppose  there is an embedding $ \Phi:\cstu(X)\to \cstu(Y)$ so that $\Phi(\ell_\infty(X))$  is a coseparable Cartan subalgebra of   $\cstu(Y)$ and $\Phi(\cstu(X))$ is   strongly-cobounded in $\cstu(Y)$. Then:
\begin{enumerate}
\item If $X$ has finite asymptotic dimension, so does $Y$.
\item If $X$ has property A, so does $Y$.
\item If $X$ has asymptotic property C,   so does $Y$.
\item If $X$ has straight finite decomposition complexity, so does $Y$.
\end{enumerate}   
   \end{corollary}

 In Section \ref{SectionSpaciallyImp}, we provide formulas for strongly continuous embeddings  between uniform Roe algebras (see Theorems \ref{ThmSpaciallyImpEmb} and \ref{ThmNSpaciallyImpEmb}).  Precisely,   all isomorphisms $\Phi:\cstu(X)\to \cstu(Y)$ between uniform Roe algebras are spacially implemented, i.e., there is a unitary $u:\ell_2(X)\to \ell_2(Y)$ so that $\Phi=\Ad(u)$ (\cite[Lemma 3.1]{SpakulaWillett2013}).
 This  was later generalized for embeddings $\Phi:\cstu(X)\to \cstu(Y)$ onto a hereditary subalgebra of $\cstu(Y)$ (\cite[Lemma 6.1]{BragaFarahVignati2019}). Notice that   spacially implemented embeddings are automatically strongly continuous and rank-preserving. We generalize the two results above by showing that an embedding $\cstu(X)\to \cstu(Y)$ is spacially implemented if and only if it is strongly continuous and rank-preserving, see Theorem \ref{ThmSpaciallyImpEmb} (two other characterizations of such embeddings in terms of compact operators are also given). Moreover, we prove an equivalent result for strongly continuous embeddings $\cstu(X)\to \cstu(Y)$ which send rank $1$ operators to rank $n$ operators, where $n\in \N\cup\{\infty\}$, see   Theorem \ref{ThmNSpaciallyImpEmb}.

 We finish the paper stating some natural  problems which are left open; see Section \ref{SectionProblems}.

\section{Preliminaries}\label{SecPrelim}
\subsection{Uniform Roe algebras}\label{SubsectionURA} Given  a Hilbert space $H$, $\cB(H)$ denotes the \cstar-algebra of bounded operators on $H$ and $\cK(H)$ denotes its ideal of compact operators. Given a set $X$, $(\delta_x)_{x\in X}$ denotes the standard unit basis of the Hilbert space $\ell_2(X)$ and $1_X$ denotes the identity on $\ell_2(X)$. Given $x,y\in X$, $e_{xy}$  denotes the partial isometry in $\cB(\ell_2(X))$ so that  $e_{xy}\delta_x=\delta_y$ and $e_{xy}\delta_z=0$ for all $z\neq x$. Given $A\subset X$, we write $\chi_A=\SOTh\sum_{x\in X}e_{xx}$; so $1_X=\chi_X$. 

If $(X,d)$ is a metric space,   $a\in \cB(\ell_2(X))$ and $r>0$, we say that the \emph{propagation of $a$ is at most $r$}, and write $\propg(a)\leq r$, if 
\[d(x,y)> r\ \text{ implies }\ \langle a\delta_x,\delta_y\rangle=0, \ \text{ for all }\ x,y\in X.\]
The \emph{support of $a$} is defined by 
\[\supp(a)=\Big\{(x,y)\in X\times X\mid \langle a\delta_x,\delta_y\rangle\neq 0\Big\}.\]

Given a metric space $(X,d)$, the subset of all operators in $\cB(\ell_2(X))$ with finite propagation is a $*$-algebra called the \emph{algebraic uniform Roe algebra of $X$},   denoted by  $\cstu[X]$.   The closure of $\cstu[X]$ is a \cstar-algebra called the \emph{uniform Roe algebra of $X$}, denoted by $\cstu(X)$.

Instead of presenting the original definition of property A, we give a  definition in terms of ghost operators which is better suited for our goals: 

\begin{definition}\label{DefiPropA}
Let $X$ be a uniformly locally finite metric space. \begin{enumerate}
\item An operator $a\in \cB(\ell_2(X))$ is called a \emph{ghost} if for all $\eps>0$ there is a finite $A\subset X$ so that 
\[|\langle a\delta_x,\delta_y\rangle |\leq \eps \ \text{ for all }\ x,y\in X\setminus A.\]
\item The metric space $X$ has \emph{property A} if all ghost operators in $\cstu(X)$ are compact.\footnote{This is not  the original definition of property A, but it is equivalent to it by \cite[Theorem 1.3]{RoeWillett2014}. For its original definition, we refer the reader to \cite[Definition 2.1]{Yu2000}. }
\end{enumerate}
\end{definition}

We  recall the notions of coarse-likeness  introduced in \cite{BragaFarah2018}.

\begin{definition}
Let $X$ and $Y$ be a metric space.
\begin{enumerate}
\item Given $\eps>0$ and $r>0$, an operator $a\in \cB(\ell_2(X))$ can be  \emph{$\eps$-$r$-approximated} (equivalently, $a$ is \emph{$\eps$-$r$-approximable})  if there is $b\in \cstu(X)$ with $\propg(b)\leq r$ so that $\|a-b\|\leq \eps$. 
\item A map $\Phi:\cstu(X)\to \cstu(Y)$ is \emph{coarse-like} if for all $\eps,s>0$ there is $r>0$ so that $\Phi(a)$ can be $\eps$-$r$-approximated for all contractions $a\in \cstu(X)$ with $\propg(a)\leq s$.
\end{enumerate}
\end{definition}

The following is a simple consequence of \cite[Lemma 4.9]{BragaFarah2018} (see  \cite[Proposition 3.3]{BragaFarahVignati2019} for a precise proof; cf. \cite[Theorem 4.4]{BragaFarah2018}).\footnote{The hypothesis of \cite[Proposition 3.33]{BragaFarahVignati2019} actually demand $\Phi$ to be     compact-preserving. However, this is so as its proof  used  an earlier version of \cite[Lemma 4.9]{BragaFarah2018} which required compactness. The  (newer) published version of \cite[Lemma 4.9]{BragaFarah2018} does not do so, hence the proof of \cite[Proposition 3.33]{BragaFarahVignati2019} holds for  noncompact  preserving $\Phi$'s.}

\begin{theorem}\label{ThmCoarseLike}
Let $X$ and $Y$ be uniformly locally finite metrics spaces and $\Phi:\cstu(X)\to \cstu(Y)$ be a strongly continuous linear operator. Then $\Phi$ is coarse-like.
\end{theorem}
 
\subsection{Roe Cartan pairs}\label{SubsectionRoeCartan}
Recall, given a \cstar-algebra $A$, a \cstar-subagebra $B\subset A$ is called a \emph{Cartan subalgebra of $A$} if
\begin{enumerate}
\item $B$ is a maximal abelian self-adjoint subalgebra of $A$,
\item $B$ contains an approximate unit for $A$, 
\item the normalizer of $B$ in $A$, i.e., $\{a\in A\mid aBa^*,\subset a^*Ba\subset B\}$, generates $A$ as a \cstar-algebra, and 
\item there is a faithful condition expectation $A\to B$.
\end{enumerate}
If $X$ is a u.l.f. metric space, then $\ell_\infty(X)$ is a Cartan subalgebra of $\cstu(X)$ \cite[Proposition 4.10]{WhiteWillett2017}.

Let $A$ be a unital \cstar-algebra and $B\subset A$ be a Cartan subalgebra of $A$. We say that $(A,B)$ is \emph{Roe Cartan pair} if 
\begin{enumerate}
\item \label{ItemRoeCartPair1} $A$ contains the algebra of compact operators on a infinite dimensional Hilbert space as an essential ideal,
\item\label{ItemRoeCartPair2} $B$ is isomorphic to the \cstar-algebra $\ell_\infty(\N)$, and
\item $B$ is \emph{co-separable} in $A$, i.e., there is a countable $S\subset A$ so that $A=\cst(B,S)$.
\end{enumerate}
Roe Cartan pairs were recently introduced to the literature in \cite{WhiteWillett2017}. Notice that, as $\ell_\infty(X)$ is a Cartan subalgebra of $\cstu(X)$, it is clear that $(\cstu(X), \ell_\infty(X))$ is a Roe Cartan pair for any u.l.f. metric space $X$. Moreover, we notice that it is not known whether co-separability is a necessary property. Precisely, if $(\cstu(X),B)$ satisfies  \eqref{ItemRoeCartPair2}, it is not known whether $B$ is automatically co-separable. By \cite[Corollary 6.3]{BragaFarahVignati2020}, this is indeed the case if $X$ has property A.

The importance of Roe Cartan pairs to our goals lies in the following theorem:
 
\begin{theorem}[Theorem B of \cite{WhiteWillett2017}]
Let $(A,B)$ be a Roe Cartan pair. Then there is a u.l.f.  metric space $X$ and an isomorphism $\Phi:A\to \cstu(X)$ so that $\Phi(B)=\ell_\infty(X)$. \label{WhiteWillettRoeCartan}
\end{theorem}

 \subsection{Coarse geometry}
Given a metric space $(X,d)$, $x\in X$ and $\eps>0$, we denote  by $B(x,\eps)$ the closed unit ball centered at $x$ of radius $\eps$.  Given a subset $A\subset X$ and $K>0$, we write 
\[A^K=\Big\{x\in X\mid d(x,A)\leq K\Big\}.\]
A metric space $(X,d)$ is called \emph{uniformly locally finite}, \emph{u.l.f.} for short, if $\sup_{x\in X}|B(x,r)|<\infty$ for all $r>0$, where $|B(x,r)|$ denotes the cardinality of $B(x,r)$.

Let $(X,d)$ and $(Y,\partial)$ be metric spaces and $f:X\to Y$ be a map.  The \emph{modulus of uniform continuity of $f$} is defined by 
\[\omega_f(t)=\Big\{\partial(f(x),f(y))\mid d(x,y)\leq t\Big\}\]
for $t\geq 0$. Then $f$ is called \emph{coarse} if $\omega_f(t)<\infty$ for all $t\geq 0$. Given another map $g:X\to Y$, we say that $f$ is \emph{close} to $g$, and write $f\sim g$, if 
\[\sup_{x\in X}\partial (f(x),g(x))<\infty.\]
The map $f$ is a \emph{coarse equivalence} if $f$ is coarse and there is a coarse map $h:Y\to X$ so that $f\circ h\sim \mathrm{Id}_Y$ and $h\circ f\sim \mathrm{Id}_X$.

Given $K>0$, we say that $f$ is \emph{$K$-co-coarse} if for all $\eps>0$ there is $\delta>0$ so that 
 \[B(f(x),\eps)\subset f\big(B(x,\delta)\big)^K\]
 for all $x\in X$. The map  $f$ is \emph{co-coarse} if it is \emph{$K$-co-coarse} for some $K>0$. If $f$ is both coarse and co-coarse, $f$ is a \emph{coarse quotient}.

\begin{proposition}\label{PropCoarseQuoCoarseProp}
Let $(X,d)$ and $(Y,\partial)$ be metric spaces and $f,g:X\to Y$ be close maps. If $f$ is a coarse quotient, so is $g$.
\end{proposition}

 \begin{proof}
Coarseness is well-known to be preserved under closeness. Moreover, if  $m=\sup_{x\in X}\partial(f(x),g(x))$ and $K>0$ is so that $f$ is $K$-co-coarse, then it is straightforward to check that $g$ is $(K+m)$-co-coarse.
 \end{proof}
 
From now on, we assume  throughout the paper that $d$ and $\partial$ are the metrics on $X$ and $Y$, respectively. 
 
\subsection{Examples of coarse quotients}\label{SubsectionCoarseQuotientMapsEX}
Firstly, notice that every coarse equivalence is a coarse quotient map (\cite[Proposition 2.5]{Zhang2015Israel}). But coarse equivalences are far from the only examples of coarse quotient maps. For instance, given $n,m\in\N$ with $n<m$, the standard projection $q:\Z^m\to \Z^n$ is clearly a coarse quotient map (and clearly not a coarse equivalence). The map $f:\Z\to \N$ given by  $f(n)=2n$ for $n\in\N$ and $f(n)=-2n+1$ for all $n\in \Z\setminus \N$ is also a coarse quotient. In this case, $f$ is a bijective coarse quotient  which is not a coarse embedding/equivalence. Similar constructions give us bijective coarse quotients $\Z^n\to \N^n$ for all $n\in\N$.

More generally, group actions   give us natural examples of coarse quotient maps. Precisely, let $(X,d)$ be a discrete metric space and $G$ be a group acting on   $X$. We say that $G$ acts on $X$  \emph{uniformly by coarse equivalences} if each $g\in G$ acts on $X$ by a coarse equivalence and if there is a map $\omega:[0,\infty)\to [0,\infty)$ so that 
\[d(g\cdot x,g\cdot y)\leq \omega(d(x,y))\]
for all $x,y\in X$ and all $g\in G$. If $G$ is a finite group which acts on $X$ by coarse equivalences, then it is automatic that $G$ acts on $X$  uniformly by coarse equivalences.

Given $G\curvearrowright X$, denote  the orbit space of this action by $X/G$, i.e., define an equivalence $\sim_G$ on $X$ by letting $x\sim_G y$ if there is $g\in G$ with $g\cdot y=x$ and $X/G$ is the set of  $\sim_G$-equivalence classes. The coarse structure of $X$ induces a coarse structure on $X/G$. Precisely, we can endow $X/G$ with the metric  
\[\partial([x],[y])=\min\Big\{\sup_{x'\in [x]}\inf_{y'\in [y]} d(x',y'),\sup_{y'\in [y]}\inf_{x'\in [x]} d(x',y')\Big\}\]
for all $[x],[y]\in X/G$. As $X$ is discrete, $\partial$ is clearly a metric on $X/G$. Let $q:X\to X/G$ be the quotient map. If $G$ acts on $X$ uniformly by coarse equivalences, then it is straightforward to check that $q$ is a coarse quotient map. The coarse geometry of those spaces was   studied in \cite{HigginbothamWeighill2019JTopAnal}.

\section{Coarse quotients and geometry preservation}\label{SectionGeoPreser}

\emph{Coarse properties} are those mathematical properties   of metric spaces which are preserved under coarse equivalence. Some large scale properties are also stable under coarse embeddings, i.e., if $X$ coarsely embeds into $Y$ and $Y$ has such property, then so does $X$. Moreover, for u.l.f. metric spaces, it is also known that some large scale properties are stable under the existence of uniformly finite-to-one coarse maps  --- for instance, this holds for property A, asymptotic dimension and finite dimension decomposition, see  \cite[Proof of Corollary 1.3, Proposition 2.5, and Corollary 5.8]{BragaFarahVignati2019}.

In this section, we show that the ``opposite'' phenomena can happen for uniformly finite-to-one coarse \emph{quotient} maps $ X\to Y$. Precisely, we show that the existence of such maps is often enough so that  large scale geometric properties of $X$ passes to $Y$. This is the case for finite asymptotic dimension, straight finite decomposition complexity, asymptotic property C, and property A (Corollary \ref{CorCoarseQuoGROPRESERVATION}).  

We start noticing that compositions of coarse quotients are coarse quotients. This is essentially done in \cite[Proposition  2.5]{Zhang2015Israel}, but we include a  short proof for the reader's convenience.

 \begin{proposition}\label{PropCompositionCoarseQuo}
 Let $X$, $Y$ and $Z$ be metric spaces, and $f:X\to Y$ and $g:Y\to Z$ be a coarse quotient maps. Then $g\circ f$ is a coarse quotient. In particular, $g\circ f$ is a coarse quotient map given that $f$ is a coarse quotient map and $g$ is a coarse equivalence.
 \end{proposition}
 
 \begin{proof}
If both $f$ and $g$ are coarse quotient maps, both are coarse and so is their composition. We are left to show that $g\circ f$ is co-coarse. For that, fix $K>0$ so that $f$ and $g$ are  $K$-co-coarse. Given $\eps>0$, let $\delta>0$ be so that $B(g(y),\eps)\subset g(B(y,\delta))^K$ for all $y\in Y$, and let $\gamma>0$ be so that $B(f(x),\delta)\subset f(B(x,\gamma))^K$ for all $x\in X$. Then, for $L=K+\omega_g(K)$, we have 
\[
B(g\circ f(x),\eps)\subset g(B(f(x),\delta))^K
\subset g(f(B(y,\gamma))^K)^K
\subset g(f(B(y,\gamma)))^{L}
\]
for all $x\in X$. Hence the assignment $\eps\mapsto \gamma$ witness that $g\circ f$ is $L$-co-coarse.  
 \end{proof}

Our main tool in order to obtain coarse geometry preservation, is based on the next concept. This was introduced in   \cite[Subsection 3.2]{MiyataVirk2013Fundamenta} as property ``$B_n$'' and it is the ``coarse version'' of a function   being $n$-to-1.

\begin{definition}\label{DefiCoraselyNto1}
Let $X$ and $Y$ be metric spaces. Given $f:X\to Y$ and $n\in\N$, we say that $f$  is  \emph{coarsely $n$-to-1} if for each $s>0$ there is $r>0$ so that for all $B\subset Y$, with $\diam(B)\leq s$, there are $A_1,\ldots, A_n\subset X$, with $\diam(A_i)\leq r$ for all $i\in \{1,\ldots, n\}$, so that $f^{-1}(B)\subset \bigcup_{i=1}^nA_i$. The map $f:X\to Y$ is called \emph{uniformly coarsely finite-to-one} if $f$ is coarsely $n$-to-1 for some $n\in\N$.
\end{definition}

The next proposition is the main result of this section.

\begin{proposition}\label{PropCoarselyNto1}
Let $X$ be a metric space and $Y$ be a u.l.f. metric space. Any  uniformly finite-to-one coarse quotient map $X\to Y$ is   uniformly coarsely finite-to-one.
\end{proposition}

We need a couple of preliminary results before proving Proposition \ref{PropCoarselyNto1}. We start by noticing  that ``coarsely $n$-to-1'' is a coarse property for maps between metric spaces. 
 
\begin{proposition}\label{PropCoarselyNto1CoarseProperty}
 Let $X$ and $Y$ be metric spaces and $f,g:X\to Y$ be close maps. Given $n\in\N$, if $f$ is coarsely $n$-to-1, so is $g$.
\end{proposition}
 
\begin{proof}
As $f$ and $g$ are close, let $k=\sup_{x\in X}\partial(f(x),g(x))<\infty$. Given $s>0$,  let $r>0$ be as in Definition \ref{DefiCoraselyNto1} for  $s+2k$, $f$  and $n$.  Fix  $B\subset Y$ with $\diam (B)\leq s$. As $\diam(B^k)\leq s+2k$,  our choice of $r$ gives  $A_1,\ldots, A_n\subset X$, with $\diam(A_i)\leq r$ for all $i\in \{1,\ldots, n\}$,  so that $f^{-1}(B^k)\subset \bigcup_{i=1}^nA_i$.  Then $g^{-1}(B)\subset \bigcup_{i=1}^nA_i$.
\end{proof}

 \begin{lemma}\label{LemmaInjCoaQuo}
  Let $X$ be a metric space and $Y$ be a u.l.f.  metric space. Any  injective coarse quotient  $X\to Y$ is  uniformly coarsely finite-to-one. 
 \end{lemma}
 
 \begin{proof}
 Let $f:X\to Y$ be an injective  coarse quotient map. Without loss of generality, assume that $f$ is surjective.  Fix  $K>0$ so that $f$ is $K$-co-coarse and let $m=\sup_{y\in Y}|B(y,K)|$. Let us show that $f$ is coarsely $m$-to-$1$. Fix $\eps>0$. By our choice of $K$, there is   $\delta=\delta(\eps)>0$ so that 
 \[B(f(x),2\eps)\subset f\big(B(x,\delta)\big)^K\]
 for all $x\in X$.

Fix $y\in Y$. We construct a finite sequence $(x_i)_i$ of elements of $X$ by induction as follows. Pick $x_1\in X$  so that $f(x_1)\in B(y,\eps)$. Let $\ell\in\N$ and assume that  $x_1,\ldots, x_\ell\in X$ have been chosen. If \[B(y,\eps)\subset \bigcup_{i=1}^\ell f\big(B(x_i,3\delta)\big),\]
we stop the procedure and $(x_i)_{i=1}^\ell$ is the outcome of it. If not, then pick $x_{\ell+1}\in X$ so that \[f(x_{\ell+1})\in  B(y,\eps)\setminus \bigcup_{i=1}^\ell f\big(B(x_i,3\delta)\big).\]
This completes the induction. As $B(y,\eps)$ contains finitely many elements, this procedure is finite. 

Let $(x_i)_{i=1}^\ell$ be the finite sequence obtained by the procedure above. Let us show that $\ell\leq m$. Assume by contradiction that   $\ell>m$  and let $z=f(x_{m+1})$. As each $f(x_i)$ belongs to $B(y,\eps)$, we have that   $\partial (z,f(x_i))\leq 2 \eps$ for all $i\in \{1,\ldots, \ell\}$. Hence, by our choice of $\delta$, there is a finite sequence $(w_i)_{i=1}^m$ in $X$ so that $d(x_i,w_i)\leq \delta$ and  $\partial (z,f(w_i))\leq K$ for all $i\in\{1,\ldots, m\}$.  By the construction of $(x_i)_{i=1}^\ell$,  $f(x_i)\not\in f(B(x_j,3\delta))$ for $j<i$. Hence, $d(x_i,x_j)> 3\delta$ for all $i\neq j$, which implies that  $(w_i)_{i=1}^m$ is a distinct sequence.  As $f$ is injective,  $(f(w_i))_{i=1}^m$ is a distinct sequence. As $z\not\in  \bigcup_{i=1}^m f(B(x_i,3\delta))$, 
$\{z,f(w_1),\ldots,f(w_m)\}$ is a subset of $B(z,K)$ with $m+1$ elements. This contradicts our choice of $m$.

As $\ell=m$, $B(y,\eps)\subset \bigcup_{i=1}^m f(B(x_i,3\delta))$. As $f$ is injective, this implies that 
\[f^{-1}(B(y,\eps))\subset \bigcup_{i=1}^m B(x_i,3\delta).\]
As $y$ was arbitrary, the assignment $\eps\mapsto 3\delta(\eps)$ witness that $f$ is a coarsely $m$-to-1 map.   
 \end{proof}

\begin{lemma}\label{LemmaMakeInj}
Let $X$ and $Y$ be metric spaces, and $f:X\to Y$ be a uniformly finite-to-one map. If $Y$ is u.l.f., then there is a u.l.f. metric space $Z$, with $Y\subset Z$, and an injective map $g:X\to Z$ which is close to $f$. Moreover, the inclusion $Y\hookrightarrow Z$ is a coarse  equivalence.
\end{lemma} 
 
 \begin{proof}
 Let $n=\sup_{y\in Y}|f^{-1}(y)|$. Let $Z=Y\times \{1,\ldots, n\}$ and define a metric $\partial_Z$ on $Z$ by letting 
 \[\partial_Z((y,i),(z,j))=\partial(y,z)+1\]
 for all distinct  $(y,i),(z,j)\in Z$. As $(Y,\partial)$ is u.l.f., so is $(Z,\partial_Z)$.  For each $y\in Y$,  enumerate $f^{-1}(y)$, say $f^{-1}(y)=\{x_1^y,\ldots, x_{i(y)}^y\}$. As $X=\bigsqcup_{y\in Y}\{x_1^y,\ldots, x_{i(y)}^y\}$, we can define a map $g:X\to Z$ by letting $g(x^y_j)=(y,j)$ for all $y\in Y$ and all $j\in \{1,\ldots, i(y)\}$. It is clear that $f$ is close to $g$ and that $Y\times \{1\}\hookrightarrow Z$ is a coarse equivalence. By identifying $Y$ with $Y\times\{1\}$, we can assume that $Y\subset Z$.
 \end{proof}

\begin{proof}[Proof of Proposition \ref{PropCoarselyNto1}]
Let $f:X\to Y$ be a uniformly finite-to-one coarse quotient map. Let $Z$ and $g$ be given Lemma \ref{LemmaMakeInj} applied to $f:X\to Y$.  As the inclusion $i:Y\hookrightarrow Z$ is a coarse equivalence,    $f=i\circ f:X\to Z$ is a coarse quotient map (Proposition \ref{PropCompositionCoarseQuo}). Hence, as $f$ is close to $g$,  $g$ is a coarse quotient map (Proposition  \ref{PropCoarseQuoCoarseProp}). As $g$ is injective,   Lemma \ref{LemmaInjCoaQuo} gives us that $g$ is  uniformly  coarsely finite-to-one. Using that $f$ and $g$ are close to each other once  again, this shows that $f$ is  uniformly coarsely  finite-to-one (Proposition \ref{PropCoarselyNto1CoarseProperty}). 
\end{proof} 
 
 We can now use Proposition \ref{PropCoarselyNto1} in order to obtain that some coarse properties of $X$ passes to $Y$ in the presence of a uniformly finite-to-one coarse quotient map $X\to Y$. For brevity, we do not introduce the  geometric properties mentioned in the corollary below. Instead, we refer the reader to   \cite[Section 1.E]{Gromov1993LectureNotes} for the definition of asymptotic dimension, \cite[Definitions 	   2.7.7]{NowakYuBook} for the definition of   asymptotic property C, and   \cite[Definition 2.2]{DranishnikovZarichnyi2014TopAppli} for the definition of straight finite decomposition (property A has been defined in Definition \ref{DefiPropA}).
  
 \begin{corollary}\label{CorCoarseQuoGROPRESERVATION}
 Let $X$ and $Y$ be metric spaces, assume that $Y$ is  u.l.f. and that there is a uniformly finite-to-one coarse quotient map $X\to Y$.  The following holds:
 
 \begin{enumerate}
 \item\label{ItemCorCoarseQuoGROPRESERVATION1} If $X$ has finite asymptotic dimension, so does $Y$.
 \item\label{ItemCorCoarseQuoGROPRESERVATION2} If $X$ has property A, so does $Y$.
 \item\label{ItemCorCoarseQuoGROPRESERVATION3} If $X$ has asymptotic property C, do does $Y$.
 \item\label{ItemCorCoarseQuoGROPRESERVATION4} If $X$ has straight finite decomposition complexity, so does $Y$.
 \end{enumerate}
\end{corollary} 

\begin{proof} 
 \eqref{ItemCorCoarseQuoGROPRESERVATION1} This follows from Proposition \ref{PropCoarselyNto1} and   \cite[Theorem 1.4]{MiyataVirk2013Fundamenta}.

 \eqref{ItemCorCoarseQuoGROPRESERVATION2} This follows from Proposition \ref{PropCoarselyNto1} and  and   \cite[Corollary  7.5]{DydakVirk2016RevMatComp}.

 \eqref{ItemCorCoarseQuoGROPRESERVATION3} This follows from Proposition \ref{PropCoarselyNto1} and  \cite[Theorem 6.2]{DydakVirk2016RevMatComp}.

 \eqref{ItemCorCoarseQuoGROPRESERVATION4} This follows from Proposition \ref{PropCoarselyNto1} and   \cite[Theorem 8.4 and Theorem 8.7]{DydakVirk2016RevMatComp}.
\end{proof}

	\section{Cobounded embeddings between uniform Roe algebras}\label{SectionCoboundedEmb} 

In this section, we study embeddings into uniform Roe algebras whose images are ``large". We obtain that such embeddings can often be enough to guarantee the existence of coarse quotient maps between the spaces.  Theorems \ref{ThmIFFNoGeoProp}, \ref{ThmEmbIFFCoarseQuotient}, \ref{ThmEmbCoarseQuotientGHOST}, and Corollary \ref{CorFiniteAsDim} are proven in this section.

The next proposition is well known and its proof can be found for instance in \cite[Proposition 2.4]{BragaFarah2018}. Recall, given $K>0$, a subset $A$ of a metric space $X$ is called \emph{$K$-separated} if $d(x,y)\geq K$ for all distinct $x,y\in A$.

\begin{proposition}\label{PropositionULFPartition}
Let $X$ be a u.l.f. metric space. Given any $K>0$, there is $n\in\N$ and a partition
\[X=X_1\sqcup X_2\sqcup\ldots\sqcup X_n\]
so that each $X_i$ is $K$-separated.
\end{proposition}

\begin{proposition}\label{PropBijCoarQuotCobounEmb}
Let $X$ and $Y$ be u.l.f. metric spaces, and $f:X\to Y$ be an injective map so that $f:X\to f(X)$ is a coarse quotient map. Then there is a spacially implemented  embedding $\Phi:\cstu(X)\to \cstu(Y)$ so that $\Phi(\ell_\infty(X))=\ell_\infty(f[X])$ and $\Phi(\cstu(X))$ is cobounded in $\cstu(f[X])$.
\end{proposition}

\begin{proof} 
  Let $u_f:\ell_2(X)\to \ell_2(Y)$ be the isometric embedding  given by $u_f\delta_x=\delta_{f(x)}$ for all $x\in X$. Then it is easy to see that $\Phi=\Ad(u_f):\cstu(X)\to \cstu(Y)$ is an embedding  (see \cite[Theorem 1.2]{BragaFarahVignati2019} for a detailed proof). Since it is clear that $\Phi(\ell_\infty(X))=\ell_\infty(f[X])$, we only need to notice that $\Phi(\cstu(X))$ is cobounded in $\cstu(f[X])$.

Fix $K>0$ so that $f$ is $K$-co-bounded and let $Z=f[X]$.  Fix  $\eps>0$ and let  $a\in \cstu[Z]$ with $\propg(a)\leq \eps$. Without loss of generality, assume $\eps>K$. By our choice of $K$, there is $\delta>0$ so that 
\[B(f(x),\eps)\subset f\big(B(x,\delta)\big)^K\]
for all $x\in X$. As $Y$ is u.l.f., there is $n=n(\eps,Y)\in \N$ and a partition  $Z=\bigsqcup_{i=1}^nZ_i$ so that  each $Z_i$ is $3\eps$-separated (Proposition \ref{PropositionULFPartition}). For each $(i,j)\in \{1,\ldots, n\}^2$, let  $a(i,j)=\chi_{Z_i}a\chi_{Z_j}$.

Fix $i,j\in \{1,\ldots, n\}$. For simplicity, let $b=a(i,j)$. Clearly, $\propg(b)\leq \propg(a)\leq \eps$. Hence, as $Z_i$ and $Z_j$ are $3\eps$-separated, there are $3\eps$-separated sequences $(y_n)_n$ and $(y'_n)_n$ in $Z$ so that \[b=\sum_{n\in\N}b_ne_{y_ny'_n}\ \text{ where }\ b_{n}=\langle b\delta_{y_n},\delta_{y'_n}\rangle\ \text{ for all }\ n\in\N.\]
In particular, $\partial(y_n,y'_n)\leq \eps$ for all $n\in\N$. As $f:X\to Z$ is bijective, fix a sequence $(x_n)_n$ of distinct elements in $X$ so that $f(x_n)=y_n$ for all $n\in\N$. By our choice of  $\delta$, for each $n\in\N$, there are $z_n \in X$ so that $d(x_n,z_n)\leq \delta$ and $\partial (f(z_n),y'_n)\leq K$.  As $(y'_n)_n$ is $3\eps$-separated and $K< \eps$, it follows that $(f(z_n))_n$ is a sequence of distinct elements and, as $f$ is injective, so is $(z_n)_n$. In particular,  $c=\SOTh\sum_{n\in\N}b_ne_{x_nz_n}$ is well defined and, as $\propg(c)\leq \delta$,  $c\in \cstu(X)$. Similarly,     $d=\SOTh\sum_{n\in\N}e_{f(z_n)y'_n}$ is also well defined and it has propagation at most $K$. 

Notice that $b=d\Phi(c)$. As $\eps>0$ and $i,j\in \{1,\ldots,n\}$ were arbitrary, we are done. 
\end{proof}

 \begin{remark}\label{RemarkCobStCob}
Notice that if $f:\Z\to \N$ is the bijective coarse quotient given by $f(n)=2n$ for $n\in\N$ and $f(n)=-2n+1$ for $n\in\Z\setminus \N$, then $\Phi=\Ad(u_f)$ is actually strongly-cobounded. Indeed, let $I$ and $P$ denote the odd and even natural numbers, respectively. Then any $a\in \cstu(\N)$ can be written as \[a=\chi_Ia\chi_I+\chi_Pa\chi_P+\chi_Pa\chi_I+\chi_Ia\chi_P.\]
 Clearly, $\chi_Ia\chi_I,\chi_Pa\chi_P\in \Phi(\cstu(\Z))$. Let $g:P\to I$ be the bijection given by $g(n)=n-1$ for all $n\in P$, and let $c=\SOTh\sum_{n\in P}e_{ng(n)}$; so $\propg(c)=1$. Moreover, it is clear that  $c\chi_Pa\chi_I,c^*\chi_Ia\chi_P\in \Phi(\cstu(\Z))$. As $\chi_Pa\chi_I=c^*c\chi_Pa\chi_I$ and $\chi_Ia\chi_P=cc^*\chi_Pa\chi_I$, it easily follows that $\Phi(\cstu(\Z))$ is strongly-cobounded in $\cstu(\N)$. 
 
  We do not know if an arbitrary bijective coarse quotient $X\to Y$ is enough to produce an embedding $\Phi:\cstu(X)\to \cstu(Y)$ with strongly-cobounded range (see Problem \ref{ProbCoaQuotSCoboundEmb}).
 \end{remark}

For the next technical lemma, we    introduce a weakening of Definition \ref{DefiCoboundedEmb}.

 \begin{definition}
Let $X$ be a metric space and let $A_1\subset A_2$ be \cstar-subalgebras of $\cstu(X)$. Given $\eps>0$, the algebra $A_1$ is called  \emph{$\eps$-almost cobounded  in $A_2$} if there is $k>0$ so that for all $b\in A_2$ there are $a_1,\ldots , a_\ell\in A_1$  and $c_1,\ldots, c_\ell \in A_2$, with $\propg(c_i)\leq k$,  so that
\[\Big\|b-\sum_{i=1}^\ell c_ia_i\Big\|\leq \eps.\] 
\end{definition}

\begin{lemma}\label{LemmaEmbEllInftTOEllInftRankPreserv}
Let $X$ and $Y$ be u.l.f. metric spaces,   $\Phi:\cstu(X)\to \cstu(Y)$ be an embedding so that $\Phi(\ell_\infty(X))\subset \ell_\infty(Y)$ and $\Phi(e_{xx})$ has rank 1 for all $x\in X$. Let $Z\subset Y$ be a subset so that 
\[\chi_Z=\SOTh\sum_{x\in X}\Phi(e_{xx}).\]
Assume that  $\Phi(\cstu(X))$ is $\eps$-almost cobounded in $\chi_Z\cstu(Y)\chi_Z$ for some $\eps\in (0,1)$. Then there exists a bijective coarse quotient map $X\to Z$.
\end{lemma} 

\begin{proof}
As $\Phi$ is a  $*$-homomorphism, the hypothesis implies that  each $\Phi(e_{xx})$ is a projection of rank 1 in $\ell_\infty(Z)$. Hence, for each $x\in X$ there is $y_x\in Z$ so that $\Phi(e_{xx})=e_{y_xy_x}$. Define a map $f:X\to Y$ by letting $f(x)=y_x$ for each $x\in X$. Notice that $Z=f(X)$. Also, the map $f$ is clearly injective, and, by \cite[Lemma 5.3]{BragaFarahVignati2019}, it is also coarse.\footnote{Notice that, although not explicit in the statement, the only assumption needed in order for \cite[Lemma 5.3]{BragaFarahVignati2019} to hold for $\Phi$ is that $\Phi(e_{xx})$ has rank 1 for all $x\in X$.}

We now show that $f:X\to Z$ is co-coarse. For that, fix $\eps\in (0,1)$ and $k\in\N$ which witness that   $\Phi(\cstu(X))$ is $\eps$-almost cobounded in $\chi_Z\cstu(Y)\chi_Z$. Assume for a contradiction that $f$ is not $k$-co-coarse. Hence, there exists $\gamma>0$ and    sequences $(x_n)_n$ and $(y_n)_n$ in $X$ so that 
\begin{enumerate}
\item\label{Item1} $\partial (f(x_n),f(y_n))\leq \gamma$ for all $n\in\N$, and
\item\label{Item2} for all $z\in X$, $\partial (f(y_n),f(z))\leq k$ implies $d(x_n,z)\geq n$.  
\end{enumerate}
Without loss of generality, by going to a subsequence, we can assume that $f(x_n)\neq f(x_m)$ and $f(y_n)\neq f(y_m)$ for all $n\neq m$. Indeed, if this is not the case, then there is an infinite $N\subset \N$ so that either $f(x_n)=f(x_m)$,  for all $n,m\in N$, or    $f(y_n)=f(y_m)$, for all $n,m\in N$. Then  \eqref{Item1} and u.l.f.-ness of $Y$ imply that, going to a further infinite subset of $N$ if necessary, we can assume that    $f(x_n)=f(x_m)$ and $f(y_n)=f(y_m)$, for all $n,m\in N$. As $f$ is injective, there is $x,y\in X$ so that $x=x_n$ and $y=y_n$ for all $n\in N$. However, \eqref{Item2} implies that $d(x,y)=d(x_n,y_n)\geq n$ for all $n\in N$; contradiction. 

As $f(x_n)\neq f(x_m)$ and $f(y_n)\neq f(y_m)$ for all $n\neq m$ in $N$, \eqref{Item1} above implies that $\sum_{n\in N}e_{f(x_n)f(y_n)}$ converges in the strong operator topology  and it belongs to $\cstu(Y)$. By our choice of 
$k$, there are $a_1,\ldots, a_\ell\in \cstu(X)$  and  $c_1,\ldots, c_\ell\in \chi_Z\cstu(Y)\chi_Z$, with $\propg(c_i)\leq k$,    so that 
\[\Big\|\chi_Z\Big(\sum_{n\in N}e_{f(x_n)f(y_n)}\Big)\chi_Z-\sum_{i=1}^\ell c_i\Phi(a_i)\Big\|\leq \eps.\]
 Hence, we have that, for all $n\in N$,
\begin{align*}
\Big\|e_{f(y_n)f(y_n)}& \Big(\sum_{i=1}^\ell c_i\Phi(a_i)\Big)e_{f(x_n)f(x_n)}\Big\|\\
&	\geq \Big\|e_{f(y_n)f(y_n)}\chi_Z\Big(\sum_{i\in\N}e_{f(x_i)f(y_i)}\Big)\chi_Ze_{f(x_n)f(x_n)}\Big\|-\eps\\
&=\|e_{f(x_n)f(y_n)}\|-\eps\\
&=1-\eps.
\end{align*}

For each $n\in N$, let $B_n=B(f(y_n),k)\cap Z$ and let $A_n=f^{-1}(B_n)$. As each $c_i$ has propagation at most $k $ and $\supp(c_i)\subset Z\times Z$, we have that $e_{f(y_n)f(y_n)}c_i=e_{f(y_n)f(y_n)}c_i\chi_{B_n}$ for all $n\in\N$ and all $i\in \{1,\ldots,\ell \}$. Therefore, as $\Phi(\chi_{A_n})=\chi_{B_n}$, we have that
\begin{align*}
\Big\|e_{f(y_n)f(y_n)} \Big(\sum_{i=1}^\ell c_i\Phi(a_i)& \Big)e_{f(x_n)f(x_n)}\Big\|\\
&	= \Big\|  \sum_{i=1}^\ell e_{f(y_n)f(y_n)}c_i\chi_{B_n}\Phi(a_i)  e_{f(x_n)f(x_n)}\Big\|\\
&=  \Big\| \sum_{i=1}^\ell e_{f(y_n)f(y_n)}c_i\Phi(\chi_{A_n})\Phi(a_i)  \Phi(e_{x_nx_n}) \Big\|\\
&\leq \max_i\|c_i\|\sum_{i=1}^\ell \|\chi_{A_n}a_ie_{x_nx_n}\|
\end{align*}
for all $n\in N$. By going to a subsequence, a simple pigeonhole argument allow us to assume that, for some $i\in \{1,\ldots, \ell\}$, we have  
\[\inf_{n\in\N}\|\chi_{A_n}a_ie_{x_nx_n}\|>0.\]
By the definition of $(A_n)_n$ and \eqref{Item2} above, we have that $\lim_nd(x_n,A_n)=\infty$. This, together with the previous inequality, contradicts the fact that $a_i\in \cstu(X)$. 
\end{proof}

\begin{theorem}\label{ThmIFFNoGeoPropGEN}
  Let $X$ and $Y$ be uniformly locally finite  metric. The following are equivalent:
\begin{enumerate}
\item \label{ItemThmIFFNoGeoPropGEN1} There is a bijective coarse quotient map $X\to Y$.
\item \label{ItemThmIFFNoGeoPropGEN2} There is an embedding $\Phi:\cstu(X)\to \cstu(Y)$ so that $\Phi(\ell_\infty(X))=\ell_\infty(Y)$ and $ \Phi(\cstu(X))$ is  cobounded in $\cstu(Y)$.
\item \label{ItemThmIFFNoGeoPropGEN3} There is $\eps\in (0,1)$ and  an embedding $\Phi:\cstu(X)\to \cstu(Y)$ so that $\Phi(\ell_\infty(X))=\ell_\infty(Y)$ and $ \Phi(\cstu(X))$ is  $\eps$-almost cobounded in $\cstu(Y)$.
\end{enumerate}  
\end{theorem}

\begin{proof}
\eqref{ItemThmIFFNoGeoPropGEN1}$\Rightarrow$\eqref{ItemThmIFFNoGeoPropGEN2} This follows from Proposition \ref{PropBijCoarQuotCobounEmb}.

\eqref{ItemThmIFFNoGeoPropGEN2}$\Rightarrow$\eqref{ItemThmIFFNoGeoPropGEN3} This is straightforward.

\eqref{ItemThmIFFNoGeoPropGEN3}$\Rightarrow$\eqref{ItemThmIFFNoGeoPropGEN1} As $\Phi$ is a  $*$-homomorphism, each $\Phi(e_{xx})$ is a projection. Hence, as the restriction $\Phi\restriction \ell_\infty(X):\ell_\infty(X)\to \ell_\infty(Y)$ is an isomorphism, $\Phi(e_{xx})$ has rank 1 for each $x\in X$. Moreover, this isomorphism also implies that $1_Y=\SOTh\sum_{x\in X}\Phi(e_{xx})$. So, the result now follows from Lemma \ref{LemmaEmbEllInftTOEllInftRankPreserv} 
\end{proof}

 \begin{proof}[Proof of Theorem \ref{ThmIFFNoGeoProp}]
 This follows from Theorem \ref{ThmIFFNoGeoPropGEN}.
 \end{proof}

 We are ready to prove Theorem \ref{ThmEmbIFFCoarseQuotient} and Theorem \ref{ThmEmbCoarseQuotientGHOST}. Since our proofs give us   slightly stronger results than the ones stated in Section \ref{SecIntro}, we need the following more general version of Definition \ref{DefiCoboundedEmb}\eqref{ItemStBound}.
	
\begin{definition}\label{DefiCoboundedEmbGEN}
Let $X$ be a metric space and let $   A_1\subset A_2$ be a \cstar-subalgebra of $\cstu(X)$. We say that  $A_1$ is \emph{strongly-cobounded in $A_2$} if there is $k\in\N$ so that for all $\eps>0$ and all contractions $b\in A_2$, there are $a_1,\ldots,a_k\in A_1$, with $\|a_i\|\leq k$,  and $c_1,\ldots, c_k\in A_2$, with $\propg(c_i)\leq k$ and $\|c_i\|\leq k$,  so that  \[\Big\|b-\sum_{i=1}^kc_ia_i\Big\|\leq \eps.\]
\end{definition}

We now prove the main theorem of this section. Theorem \ref{ThmEmbIFFCoarseQuotient} and Theorem  \ref{ThmEmbCoarseQuotientGHOST} will follow from it.

\begin{theorem}\label{ThmEmbIFFCoarseQuotientHERE}
 Let $X$ and $Y$ be u.l.f.  metric spaces, and assume that $Y$ has property A. Suppose there is an embedding $ \Phi:\cstu(X)\to \cstu(Y)$ and a hereditary \cstar-subalgebra $A$ of $\cstu(Y)$ so that $(A,\Phi(\ell_\infty(X)))$ is a Roe Cartan pair and $\Phi(\cstu(X))$ is strongly-cobounded in $A$. Then there is $Z\subset Y$ and a bijective coarse quotient  $X\to Z$.
 \end{theorem}
 
 \begin{proof}
 Fix $ \Phi:\cstu(X)\to \cstu(Y)$ and $A\subset \cstu(Y)$ as above. By hypothesis,   $(A,\Phi(\ell_\infty(X)))$ is a Roe Cartan pair, so Theorem \ref{WhiteWillettRoeCartan}  implies that  there is a u.l.f. metric space $Z$ and an isomorphism $\Psi: \cstu(Z)\to A$ so that $\Psi(\ell_\infty(Z))=\Phi(\ell_\infty(X))$. Let  $\Theta=\Psi^{-1}\circ \Phi$, so  $\Theta:\cstu(X)\to \cstu(Z)$ is an embedding so that $\Theta(\ell_\infty(X))=\ell_\infty(Z)$.

\begin{claim}
Given $\eps>0$, $\Theta(\cstu(X))$ is $\eps$-almost cobounded in $\cstu(Z)$. 
\end{claim}

\begin{proof}
Let $\eps>0$. Fix $k\in\N$ which witness  that $\Phi(\cstu(X))$ is  strongly-cobounded in $A$ and let $\delta=\eps/(k^3+1)$. As $A$ is a hereditary subalgebra of $\cstu(Y)$,  \cite[Lemma 6.1]{BragaFarahVignati2019} gives an  isometric embedding $u:\ell_2(Z)\to \ell_2(Y)$ so that $\Psi=\Ad(u)$. As $\Ad(u^*):\cstu(Y)\to \cstu(Z)$ is compact-preserving and strongly continuous,   Theorem \ref{ThmCoarseLike}  implies that   $\Ad(u^*)$ is coarse-like. Notice that $\Psi^{-1}=\Ad(u^*)\restriction A$. So, coarse-likeness gives $R>0$ so that $\Psi^{-1}(c)$ is $\delta$-$R$-approximable for all contractions $c\in A$ with   $\propg(c)\leq k$.

Let $b\in \cstu(Z)$. By our choice of $k$, there are   $a_1,\ldots, a_k\in \cstu(X)$, with $\|a_i\|\leq k$,  and $c_1,\ldots,c_k\in  A$, with $\propg(c_i)\leq k$ and $\|c_i\|\leq k$,  and so that 
\[\Big\|\Psi(b)-\sum_{i=1}^kc_i\Phi(a_i)\Big\|\leq \delta.\]
For each $i\in \{1,\ldots, k\}$, let $d_i\in \cstu(Z)$ be so that $\propg(d_i)\leq R$ and  \[\big\|\Psi^{-1}(c_i)-\|c_i\|d_i\big\|\leq \delta\|c_i\|.\] Then 

\begin{align*}
\Big\|b-&\sum_{i=1}^k\|c_i\|d_i\Theta(a_i)\Big\|\\
&\leq 
\Big\|b-\sum_{i=1}^k\Psi^{-1}(c_i\Phi(a_i))\Big\|+\Big\|\sum_{i=1}^k\Psi^{-1}(c_i) \Theta(a_i)- \sum_{i=1}^k\|c_i\|d_i\Theta(a_i)\Big\|\\
&\leq  \Big\|\Psi(b)-\sum_{i=1}^k c_i\Phi(a_i)\Big\|+\sum_{i=1}^k\big\|\Psi^{-1}(c_i)-\|c_i\|d_i\big\|\cdot \|a_i\|\\
&\leq \delta+   \delta k^3.
\end{align*}
By our choice of $\delta$, we are done.
\end{proof}

As $\Theta(\ell_\infty(X))=\ell_\infty(Z)$,  the previous claim and Theorem \ref{ThmIFFNoGeoPropGEN} imply that there is a bijective coarse quotient map $f:X\to Z$. As $Y$ has property A and $A$ is a hereditary subalgebra of $\cstu(Y)$,  \cite[Theorem 1.4]{BragaFarahVignati2019} gives us an injective coarse embedding $g:Z\to Y$. By Proposition \ref{PropCompositionCoarseQuo}, $g\circ f:X\to g(Z) $ is a coarse quotient. So we are done. 
 \end{proof}

\begin{proof}[Proof of Theorem \ref{ThmEmbIFFCoarseQuotient}]
  If $Y$ has property A, then $\cstu(Z)$ is isomorphic to $\cstu(Y)$ if and only if $Z$ and $Y$ are bijectively  coarsely equivalent \cite[Corollary 6.13]{WhiteWillett2017}.  Therefore, the result follows by running the proof  of  Theorem \ref{ThmEmbIFFCoarseQuotientHERE} for $A=\cstu(Y)$, and using \cite[Corollary 6.13]{WhiteWillett2017} instead of \cite[Theorem 1.4]{BragaFarahVignati2019} in the last paragraph of  the proof. Indeed, the map   $g:Z\to Y$ obtained at the end of the proof  of  Theorem \ref{ThmEmbIFFCoarseQuotientHERE} becomes a bijection, and so does the coarse quotient $g\circ f:X\to Y$
\end{proof}

 \begin{proof}[Proof of Theorem \ref{ThmEmbCoarseQuotientGHOST}]
If the sparse subspaces of $Y$ yield only compact ghost prohections, then an isomorphism between $\cstu(Z)$  and  $\cstu(Y)$ implies that $Y$ and $Z$ are coarsely equivalent (this follows form the proof of \cite[Theorem 1.4 and Corollary 1.5]{BragaFarahVignati2019}; equivalently, this is given by \cite[Theorem 1.3]	{BragaChungLi2019}). The result then follows by running the proof of   Theorem \ref{ThmEmbIFFCoarseQuotientHERE} for $A=\cstu(Y)$ and the result above instead of \cite[Theorem 1.4]{BragaFarahVignati2019} in the last paragraph of  the proof. 
 \end{proof}

 \begin{proof}[Proof of Corollary \ref{CorFiniteAsDim}]
 This follows straightforwardly from Theorem \ref{ThmEmbCoarseQuotientGHOST}  and Corollary \ref{CorCoarseQuoGROPRESERVATION}
 \end{proof}

 \section{Characterization of spacially implemented embeddings}\label{SectionSpaciallyImp}

The study of embeddings between uniform Roe algebras is highly dependent on the embeddings being spacially implemented or strongly continuous and rank-preserving. In this section, we prove Theorem \ref{ThmSpaciallyImpEmb} and show that those two kinds of embedding coincide. We also give two other characterization of such embeddings and prove a version of it for non rank-preserving strongly continuous embeddings (see Theorem \ref{ThmNSpaciallyImpEmb}).

For that, we need a result of \cite{BragaFarahVignati2019}. For that, recall that given a metric space $X$,  an operator $a\in \cB(\ell_2(X))$ is called \emph{quasi-local} if for all $\eps>0$ there is $S>0$ so that $d(A,B)>S$ implies $\|\chi_Aa\chi_B\|\leq \eps$ for all $A,B\subset X$. The norm closure of all quasi-local operators forms a \cstar-algebra called the \emph{quasi-local algebra of $X$}; and  this algebra is denoted by   $\cstql(X)$. Clearly, $\cstu(X)\subset \cstql(X)$ and it remains open whether this inclusion is an equality; when this is so, the metric space $X$ is called \emph{quasi-local}. We know however that this is so if $X$ has property A (\cite[Theorem 3.3]{SpakulaZhang2018}).

Before stating \cite[Theorem 4.3]{BragaFarahVignati2019}, notice that, if $\Phi:\cstu(X)\to \cstu(Y)$ is an embedding, then   $(\Phi(\chi_F))_{F\subset X, |F|<\infty}$ is an increasing net of projections. Therefore,    $p=\SOTh\sum_{x\in X} \Phi(e_{xx})$ is well defined. 

\begin{theorem}[Theorem 4.3 of \cite{BragaFarahVignati2019}]
Let  $X$ and $Y$ be  u.l.f. metric spaces,  and  $\Phi\colon \cstu(X)\to \cstu(Y)$ be an embedding. Then,   the projection $p=\SOTh\sum_{x\in X}\Phi(e_{xx})$ belong to $\cstql(Y)$ and   the map
 \[
 a\in \cstu(X)\mapsto p\Phi(a)p\in \cstql(Y)
 \] 
is a  strongly continuous embedding. Moreover, if   $\Phi$ is compact-preserving, then $ p\in \cstu(Y)$; so the map above is a strongly continuous embedding into $\cstu(Y)$.\label{ThmEmbImpliesStrongEmb}
\end{theorem}

  Theorem \ref{ThmEmbImpliesStrongEmb} allows us to obtain the following characterization of strong continuity of embeddings between uniform Roe algebras.

\begin{corollary}\label{CorEmbImpliesStrongEmb} 	
Let  $X$ and $Y$ be  u.l.f. metric spaces,  and  $\Phi\colon \cstu(X)\to \cstu(Y)$ be an embedding.  Let $p=\SOTh\sum_{x\in X}\Phi(e_{xx})$.  The following are equivalent:
\begin{enumerate}
\item \label{ItemCorEmbImpliesStrongEmb1}$\Phi$ is strongly continuous, and
\item\label{ItemCorEmbImpliesStrongEmb2} $p=\Phi(1_X)$.
\end{enumerate}
Moreover, if   either $Y$ is quasi-local or $\Phi$ is compact-preserving, then the items above are also equivalent to:
\begin{enumerate}\setcounter{enumi}{2}
\item\label{ItemCorEmbImpliesStrongEmb3} Let $A$ be the hereditary subalgebra of $\cstu(Y)$ generated by  $\Phi(\cstu(X))$. Then, $a\Phi(\cK(\ell_2(X)))=0$ implies $a=0$ for all $a\in A$.
\end{enumerate}
\end{corollary}

\begin{proof}
\eqref{ItemCorEmbImpliesStrongEmb1}$\Rightarrow$\eqref{ItemCorEmbImpliesStrongEmb2}  If $\Phi$ is strongly continuous, $p=\SOTh\sum_{x\in X}\Phi(e_{xx})=\Phi(1_X)$.

\eqref{ItemCorEmbImpliesStrongEmb2}$\Rightarrow$\eqref{ItemCorEmbImpliesStrongEmb1} This follows from Theorem \ref{ThmEmbImpliesStrongEmb}.

\eqref{ItemCorEmbImpliesStrongEmb2}$\Rightarrow$\eqref{ItemCorEmbImpliesStrongEmb3} Let  $a$ belong to $A$. If   $a\Phi(\cK(\ell_2(X)))=0$, then $a\Phi(\sum_{x\in F}e_{xx})=0$ for all finite $F\subset X$. So $ap=a\Phi(1_X)=0$. As $\Phi(1_X)$ is the unit of $A$, it follows that $a=0$.

\eqref{ItemCorEmbImpliesStrongEmb3}$\Rightarrow$\eqref{ItemCorEmbImpliesStrongEmb2} By Theorem \ref{ThmEmbImpliesStrongEmb}, if either $Y$ is quasi-local or  $\Phi$ is compact-preserving, then $p\in \cstu(Y)$. Let $q=\Phi(1_X)-p$, so $q\in\cstu(Y)$ and $q\leq \Phi(1_X)$. Hence, by the definition of $A$, we have that  $q\in A$. If $a\in \cK(\ell_2(X))$, then $a=\lim_n\chi_{X_n}a\chi_{X_n}$, where $(X_n)_n$ is an increasing sequence of finite subsets of $X$ so that $X=\bigcup_nX_n$.  Then 
\[q\Phi(a)=\lim_n\Big((\Phi(1_X)-p)\Phi(\chi_{X_n})\Phi(a\chi_{X_n})\Big)=0.\]
Therefore, by the arbitrariness of $a$ and the hypothesis,  this implies that $q=0$, i.e., $p=\Phi(1_X)$. 
\end{proof}

We point out that \eqref{ItemCorEmbImpliesStrongEmb3} of Corollary \ref{CorEmbImpliesStrongEmb} cannot be weakened   to ``$a\Phi(\cK(\ell_2(X)))=0$ implies $a=0$ for all $a\in \Phi(\cstu(X))$''. We refer   to Remark \ref{RemarkStrongEmb} for a further discussion on that.

The next result shows that spacially implemented embeddings and embeddings which are both strongly continuous and rank-preserving coincide.

 \begin{theorem}\label{ThmSpaciallyImpEmb}
Let $X$ and $Y$ be u.l.f. metric spaces and $\Phi:\cstu(X)\to \cstu(Y)$ be an embedding. The following are equivalent:
\begin{enumerate}
\item\label{ItemPropSpaciallyImpEmb1} $\Phi$ is spacially implemented by an isometric embedding $\ell_2(X)\to \ell_2(Y)$,
\item\label{ItemPropSpaciallyImpEmb2} $\Phi$ is   rank-preserving and strongly continuous,
\item\label{ItemPropSpaciallyImpEmb3}
$\Phi(1_X)\cK(\ell_2(Y))\Phi(1_X)$ is contained in   
$\Phi(\cstu(X))$, and
\item\label{ItemPropSpaciallyImpEmb4} there is a subspace $H'\subset H$ so that $\cK(H')\subset \Phi(\cstu(X))$ and $\cK(H')$ is an essential ideal of the hereditary subalgebra generated by $\Phi(\cstu(X))$.
\end{enumerate}
\end{theorem}

 \begin{proof}  The implication \eqref{ItemPropSpaciallyImpEmb1}$\Rightarrow$\eqref{ItemPropSpaciallyImpEmb2} is straightforward. Moreover, if $\Phi=\Ad(u)$ for an isometric embedding $u:\ell_2(X)\to \ell_2(Y)$, then \[	
 u\cK(\ell_2(X))u^*= \cK(\mathrm{Im}(u)) =uu^*\cK(\ell_2(Y))uu^*=\Phi(1_X)\cK(\ell_2(Y))\Phi(1_X).\]
 So the implication    \eqref{ItemPropSpaciallyImpEmb1}$\Rightarrow$\eqref{ItemPropSpaciallyImpEmb3} follows. As $\Phi(\cstu(X))\subset \cB(\mathrm{Im}(u))$,   $\cK(\mathrm{Im}(u))$ is an essential ideal  of $\Phi(\cstu(X))$. So
the implication      \eqref{ItemPropSpaciallyImpEmb1}$\Rightarrow$\eqref{ItemPropSpaciallyImpEmb4} follows. 

\eqref{ItemPropSpaciallyImpEmb2}$\Rightarrow$\eqref{ItemPropSpaciallyImpEmb1}   As $\Phi$ is rank-preserving, $\Phi(e_{xx})$ is a rank 1 projection for each $x\in X$. So, for  each $x\in X$, pick a normalized  $\xi_x\in \ell_2(X)$ so that $\Phi(e_{xx})$ is the projection onto $\Span\{\xi_x\}$. Strong continuity gives that $\Phi(1_X)=\SOTh\sum_{x\in X}\Phi(e_{xx})$. Hence, $(\Phi(e_{xx}))_n$ is a maximal set of rank 1 projections for the Hilbert space $H'=\mathrm{Im}(\Phi(1_X))$, so $(\xi_n)_n$ is an orthonormal basis for $H'$.  As $e_{xy}\in \cstu(X)$, it follows that $p_{xy}=\langle\cdot, \xi_x\rangle\xi_y\in \Phi(\cstu(X))$ for all $x,y\in X$. Hence, $\Phi(\cstu(X))$ contains $\cK(H')$. 

This implies that  $\Phi(\cK(\ell_2(X)))$ is a nontrivial ideal of $ \cK(H')$. So,  the equality $\Phi(\cK(\ell_2(X)))= \cK(H')$ must hold; hence there is a unitary $u:\ell_2(X)\to H'$ so that $\Phi(a)=uau^*$ for all $a\in \cK(\ell_2(X))$ (see \cite[Theorem 2.4.8]{MurphyBook}). On can easily check that $\Phi=\Ad(u)$ (cf. \cite[Lemma 3.1]{SpakulaWillett2013AdvMath}).

\eqref{ItemPropSpaciallyImpEmb3}$\Rightarrow$\eqref{ItemPropSpaciallyImpEmb2}  Let us notice that $\Phi(e_{xx})$ has rank 1 for all $x\in X$. Indeed, if this is not the case for some $x\in X$, there is a rank 1 projection $q<\Phi(e_{xx})$. As $q=\Phi(1_X)q\Phi(1_X)\in \Phi(\cstu(X))$,  there is a projection $q'\in \cstu(X)$ with $\Phi(q')=q$. Then $0<q'<e_{xx}$; contradiction.

As $\Phi(e_{xx})$ has rank 1 for all $x\in X$, we obtain that $\Phi(e_{xy})$ has rank 1   for all $x,y\in X$; $\Phi(e_{xy})$ is actually a partial isometry taking $\mathrm{Im}(\Phi(e_{xy}))$ onto $\mathrm{Im}(\Phi(e_{yy}))$. Therefore, $\Phi\restriction \cB(\ell_2(F))$ is rank-preserving for all finite $F\subset X$. As $\bigcup_{F\subset X,|F|<\infty}\cB(\ell_2(F))$ is dense in the finite rank operators, it easily follows that  $\Phi$ is rank-preserving. In particular, $\Phi$ is compact-preserving. 

 As $\Phi(e_{xx})\leq \Phi(1_X)$ for all $x\in  X$, we have that $p=\SOTh\sum_{x\in X}\Phi(e_{xx})\leq \Phi(1_X)$. If this is a strict inequality, pick a rank 1 projection $p_1\leq \Phi(1_X)-p$. Then $p_1=\Phi(1_X)p_1\Phi(1_X)\in \Phi(\cstu(X))$. Pick a nonzero projection $p_2\in \cstu(X)$ with $\Phi(p_2)=p_1$. Then 
 \[\|e_{xx}p_2\|=\|\Phi(e_{xx})\Phi(p_2)\|\leq\|pp_1\|=0\]
 for all $x\in X$. So $p_2=0$; contradiction. Then $p=\Phi(1_X)$ and we have $\Phi=\Ad(p)\circ \Phi$. Hence, as $\Phi$ is compact-preserving,  Theorem \ref{ThmEmbImpliesStrongEmb}  implies that $\Phi$ is strongly continuous.

\eqref{ItemPropSpaciallyImpEmb4}$\Rightarrow$\eqref{ItemPropSpaciallyImpEmb3} Fix $H'\subset H$ as in the hypothesis and  denote by $A $ the hereditary \cstar-algebra generated by $\Phi(\cstu(X))$. We claim that $\Phi(1_X)=1_{H'}$. Indeed, as $\cK(H')\subset \Phi(\cstu(X))$, all finite rank projections $q\in \cK(H')$ are below $\Phi(1_X)$. Hence, $1_{H'}\leq\Phi(1_X)$. Let $q=\Phi(1_X)-1_{H'}$, so $q\in A$. As $\cK(H')q=0$ and $\cK(H')$ is an essential ideal of $A$, $q=0$.

This shows that \[\Phi(1_X)\cK(\ell_2(Y))\Phi(1_X)=1_{H'}\cK(\ell_2(Y))1_{H'}=\cK(H')\subset \Phi(\cstu(X)),\] 
so we are done.
 \end{proof}
 
 \begin{remark}\label{RemarkStrongEmb}
As noticed in \cite[Proposition 4.1]{BragaFarahVignati2019},  there are embeddings between uniform  Roe algebras which are not strongly continuous. Moreover, the embedding can also be taken to be  rank-preserving. We recall the example: Let $X=\{n^2\mid n\in\N\}$ and  let $\cU$ be a nonprincipal ultrafilter on $X$. Let $\Phi_1:\cstu(X)\to \cstu(X)$ be the identity and let 
\[\Phi_2(\chi_A)=\left\{\begin{array}{ll}
\chi_A,& \text{ if } A\in \cU,\\
0,& \text{ if } A\not\in \cU.
\end{array}\right.\]
The map $\Phi_2$ extends to a $*$-homomorphism $\Phi_2:\cstu(X)\to \cstu(X)$ which sends $\cK(\ell_2(X))$ to zero.  The map \[\Phi=\Phi_1\oplus \Phi_2:a\in \cstu(X)\mapsto \Phi_1(a)\oplus\Phi_2(a)\in \cstu(X)\oplus \cstu(X)\]
is a rank-preserving embedding which is not strongly continuous. Moreover, as $\cstu(X)\oplus \cstu(X)\subset \cstu(X\sqcup X)$ (where $X\sqcup X$ is any metric space whose metric restricted to both copies of $X$ coincide with $X$'s original metric), this map can be considered as a map $\cstu(X)\to \cstu(X\sqcup X)$.

Notice that $\cK(\ell_2(X))\oplus\{0\}$ is an essential ideal of $\Phi(\cstu(X))$. Therefore, this example shows that  \eqref{ItemPropSpaciallyImpEmb4} cannot be weakened to ``there is a subspace $H'\subset H$ so that $\cK(H')$ is an essential ideal of $\Phi(\cstu(X))$''.
 \end{remark}

Given $n\in\N$, an operator  between operator algebras is called \emph{$1$-to-$n$ rank-preserving} if it sends operators of rank $1$ to operators of rank $n$.   We now look at strongly continuous $1$-to-$n$ rank-preserving maps between uniform Roe algebras.

 Given a metric space  $X$   and $n\in\N\cup\{\infty\}$, we have a canonical embedding  $I_n=I_{X,n}:\cstu(X)\to \cstu(X)\otimes \cB(\C^n)$. Precisely,  
\[I_n: a=[a_{xy}]\mapsto \SOTh\lim_{F\subset X, |F|<\infty} \sum_{x,y\in F}a_{xy}e_{xy}\otimes \mathrm{Id}_n\in \cstu(X)\otimes \cB(\C^n) \]
(if $n=\infty$, $\C^n$ is assume to be $\ell_2$). In other words, identifying $\cstu(X)\otimes \cB(\C^n)$ with a subalgebra of $\cB(\ell_2(X,\C^n))$ in the standard way, and thinking about operators on $\cB(\ell_2(X,\C^n))$ as $X$-by-$X$ matrices with entries in $\cB(\C^n)$, we have that $I_n([a_{xy}])=[a_{xy}\mathrm{Id}_n]$ for all $a=[a_{xy}]\in \cstu(X)$. The map $I_n$ is clearly strongly continuous.

We now prove the $n$-version of Theorem \ref{ThmSpaciallyImpEmb}. 

\begin{theorem}\label{ThmNSpaciallyImpEmb}
Let $X$ and $Y$ be u.l.f. metric spaces and $\Phi:\cstu(X)\to \cstu(Y)$ be an embedding. Given $n\in\N\cup\{\infty\}$, the following are equivalent.
\begin{enumerate}
\item\label{ItemPropNSpaciallyImpEmb1} There is an isometric embedding    $u:\ell_2(X,\C^n)\to \ell_2(Y)$ so that  $\Phi=\Ad(u)\circ I_n$.
\item\label{ItemPropNSpaciallyImpEmb2} $\Phi$ is $1$-to-$n$ rank-preserving and  strongly continuous.  
\end{enumerate}
\end{theorem}  
 
 \begin{proof}
As $I_n$ is $1$-to-$n$ rank-preserving and $\Ad(u):\cB(\ell_2(X,\C^n))\to\cB( \ell_2(Y))$ is rank-preserving,  \eqref{ItemPropNSpaciallyImpEmb1}$\Rightarrow$\eqref{ItemPropNSpaciallyImpEmb2} follows. 

\eqref{ItemPropNSpaciallyImpEmb2}$\Rightarrow$\eqref{ItemPropNSpaciallyImpEmb1}  Fix $x_0\in X$. As $\mathrm{Im}(\Phi(e_{x_0x_0}))$ has dimension $n$, fix a surjetive isometry $v:\C^n\to \mathrm{Im}(\Phi(e_{x_0x_0}))$.  Notice that $\Phi(e_{x_0x})$ is a partial isometry which takes $\mathrm{Im}(\Phi(e_{x_0x_0}))$ isometrically onto $\mathrm{Im}(\Phi(e_{xx}))$, for all $x\in X$. Define an isometric embedding  $u:\ell_2(X,\C ^n)\to \ell_2(Y)$
 by letting 
 \[u\delta_x\otimes \xi=\Phi(e_{x_{0}x})v\xi\]
 for all $x\in X$ and all $\xi\in \C^n$.  So,   $u\restriction \ell_2(\{x\},  \C^n)$ is a surjective isometry between $\ell_2(\{x\}, \C^n)$  and $\mathrm{Im}(\Phi(e_{xx}))$ for all $x\in X$. Moreover, as $\Phi$ is strongly continuous, $u$ is an  isometry onto
 \[H=\overline{\Span}\{\mathrm{Im}(\Phi(e_{xx}))\mid x\in X\}=\mathrm{Im}(\Phi(1_X)).\] 
In particular,   $u^*[H^\perp]=0$.

 Let us notice that $\Phi=\Ad(u)\circ I_n$. As $\Phi$ is strongly continuous,   it is enough to show that $\Phi(e_{xy})=\Ad(u)\circ I_n(e_{xy})$ for all $x,y\in X$. Fix $x,y\in X$. As $u^*[H^\perp]=0$ and $\Phi(e_{xy})\xi=0$ for all $\xi\in H^\perp$, it is enough to show that $\Phi(e_{xy})\xi=\Ad(u)\circ I_n(e_{xy})\xi$ for all $\xi\in H$.

 Fix $\xi\in H$.  As 
 \[H=\Big(\bigoplus_{x\in X}\mathrm{Im}(\Phi(e_{xx}))\Big)_{\ell_2},\] write $\xi=(\xi_x)_{x\in X}$ where $\xi_x\in \mathrm{Im}(\Phi(e_{xx})) $ for each $x\in X$. If $z\neq x$, $\Phi(e_{xx})\xi_z=0$ as $\Phi(e_{xx})$ and $\Phi(e_{zz})$ are orthogonal projections. As $u^*(\xi_z)\in \ell_2(\{z\}, \C^n)$, we also have that \[\Ad(u)\circ I_n(e_{xy})\xi_z=ue_{xy}\otimes \mathrm{Id}_nu^*\xi_z=0.\]
 Hence, we only need to notice that $\Phi(e_{xy})\xi_x=\Ad(u)\circ I_n(e_{xy})\xi_x$. This follows since the formula of $u$ gives that
 \[ue_{xy}\otimes \mathrm{Id}_nu^*\xi_z=\Phi(e_{x_0y})vv^*\Phi(e_{xx_0})\xi_x=\Phi(e_{x_0y})\Phi(e_{xx_0})\xi_x=\Phi(e_{xy})\xi_x,\]
 so we are done.
 \end{proof}

 \section{Questions}\label{SectionProblems}
 
We finish the paper with a selection of natural questions left unsolved. Firstly, when working with quotient maps, it is natural to look at sections of such maps.  Recall, given a surjection   $f:X\to Y$, we say that a map $g:Y\to X$ is a \emph{section of $f$} if $f\circ g=\mathrm{Id}_Y$. In the coarse category, the notion of coarse-section is more natural: precisely,  $g:Y\to X$ is a \emph{coarse-section of $f$} if it is coarse and $f\circ g\sim\mathrm{Id}_Y$. The following is straightforward. 

\begin{proposition}
Let $f:X\to Y$ be a coarse quotient between metric spaces. If $g:Y\to X$ is  a  section of $f$ which is coarse, then  $g:Y\to g(Y)$ is a bijective coarse equivalence. If $g:Y\to X$ is  a  coarse-section of $f$, then  $g:Y\to g(Y)$ is a coarse equivalence.\qed
\end{proposition}
 
Finding conditions for the existence of a coarse-section of a given coarse quotient would be interesting. More generally:
 
  \begin{problem}\label{Prop1}
 Let $X$ and $Y$ be u.l.f. metric spaces and assume that there is a uniformy finite-to-one coarse quotient $X\to Y$. Does $Y$ coarsely embed into $X$?
 \end{problem}

While Theorem \ref{ThmIFFNoGeoProp} gives us an embedding with cobounded range,  Theorem \ref{ThmEmbIFFCoarseQuotient} assumes strong-coboundedness on the embedding. Without assuming some further structure on the metric space $X$, we do not know if a bijective coarse quotient is enough to obtain an embeddings with   strongly-cobounded range (see Remark \ref{RemarkCobStCob}).
 
 \begin{problem}\label{ProbCoaQuotSCoboundEmb}
 Let $X$ and $Y$ be u.l.f. metric spaces, and $f:X\to Y$ be a bijective coarse quotient map. Is there an embedding $\Phi:\cstu(X)\to \cstu(Y)$ with strongly-cobounded range so that $\Phi(\ell_\infty(X))$ is a Cartan subalgebra of $\cstu(Y)$?
 \end{problem}

Although Corollary \ref{CorCoarseQuoGROPRESERVATION}  shows that the existence of a uniformly finite-to-one coarse quotient $X\to Y$ gives us that $Y$ has finite asymptotic dimension provided that the same holds for $X$, we do not have a precise bounded for $\mathrm{asydim}(Y)$. 
 
\begin{problem}\label{ProbAsydim}
Say $X$ and $Y$ are u.l.f. metric spaces, and assume that there is a uniformly finite-to-one coarse map $X\to Y$. Does it follows that $\mathrm{asydim}(Y)\leq \mathrm{asydim}(X)$?
\end{problem} 
 
It is known that, for proper metric spaces $X$ with    $\mathrm{asydim}(X)<\infty$, we have that $\mathrm{asydim}(X)=\mathrm{dim}(\nu X)$, where $\nu X$ is the Higson corona of $X$ (see \cite[Theorem 6.2]{Dranishnikov2000UspekhiMatNauk} for this result and  \cite[Subsection 2.3]{RoeBook}  for details on the Higson corona). Hence,  in order to give a positive answer to Problem \ref{ProbAsydim}, it would be enough to give a positive answer to the following problem:

\begin{problem} Let $X$ and $Y$ be u.l.f. metric spaces and assume that there is a bijective coarse quotient $X\to Y$. Does it follow that  $\mathrm{dim}(\nu X)= \mathrm{dim}(\nu Y)$?
  \end{problem}
  
A positive answer to the next problem would be very useful in order to extend the current rigidity results in the literature to non rank-preserving embeddings $\cstu(X)\to \cstu(Y)$.
 
  \begin{problem}
 Can $v:\C^n\to\mathrm{Im}( \Phi(e_{x_0x_0}))$ in the proof of Theorem \ref{ThmNSpaciallyImpEmb} be chosen so that $\Ad(u)(\cstu(X)\otimes \rM_n(\C))\subset \cstu(Y)$?
 \end{problem}
 
 On a different direction, but still trying   to better understand non rank-preserving embeddings $\cstu(X)\to \cstu(Y)$, a version of Theorem \ref{ThmEmbImpliesStrongEmb} would be very interesting.

 \begin{problem}
Let  $X$ and $Y$ be  u.l.f. metric spaces,  and  $\Phi\colon \cstu(X)\to \cstu(Y)$ be an embedding. Is there a projection    $p\in\cstql(Y)$ so that    the map
 \[
 a\in \cstu(X)\mapsto p\Phi(a)p\in \cstql(Y)
 \] 
is a rank-preserving embedding?
\end{problem}

\begin{acknowledgments}
The author would like to thank the anonymous referee for spotting an error in an earlier version of Corollary \ref{CorEmbImpliesStrongEmb}.
\end{acknowledgments} 
 
\newcommand{\etalchar}[1]{$^{#1}$}
\providecommand{\bysame}{\leavevmode\hbox to3em{\hrulefill}\thinspace}
\providecommand{\MR}{\relax\ifhmode\unskip\space\fi MR }
\providecommand{\MRhref}[2]{%
  \href{http://www.ams.org/mathscinet-getitem?mr=#1}{#2}
}
\providecommand{\href}[2]{#2}

\end{document}